\documentclass[a4paper,11pt]{amsart}

\usepackage[top=3.5cm,bottom=3.5cm,left=3.6cm,right=3.6cm]{geometry}
\usepackage{amsfonts}
\usepackage{amsmath}
\usepackage{amssymb}
\usepackage{amsthm}
\usepackage[T1]{fontenc}
\usepackage{graphicx}
\usepackage{hyperref}
\usepackage{enumerate}
\usepackage{mathtools}
\usepackage{enumitem}
\usepackage[titletoc]{appendix}
\numberwithin{equation}{section}

\newtheorem{dfn}{Definition}[section]
\newtheorem{propo}[dfn]{Proposition}

\newtheorem{theo}[dfn]{Theorem}

\newtheorem{cor}[dfn]{Corollary}
\newtheorem{lem}[dfn]{Lemma}

\theoremstyle{definition}
\newtheorem{nota}[dfn]{Notation}
\theoremstyle{remark}

\newtheorem{rem}[dfn]{Remark}

\newcommand{\suite}[2]{{\left(#1\right)}_{#2}}
\newcommand{\norm}[1]{\left\lVert #1 \right\rVert}
\newcommand{\ens}[1]{\left\{ #1\right\}}
\newcommand{\R}{\mathbb{R}}
\newcommand{\Z}{\mathbb{Z}}	
\newcommand{\N}{\mathbb{N}}

\newcommand{\ds}{\displaystyle}

\newcommand{\abs}[1]{\left|#1\right|}

\providecommand{\newoperator}[3]{%
  \newcommand*{#1}{\mathop{#2}#3}}
\newoperator{\re}{\mathrm{Re}}{\,}
\newoperator{\im}{\mathrm{Im}}{\,}

\begin{document}

\title{Holderian weak invariance principle for stationary mixing sequences}

\author{Davide Giraudo }

\address{Universit\'e de Rouen, LMRS, Avenue de l'Universit\'e, BP 12 76801 
Saint-\'Etienne-du-Rouvray cedex, France.}

\email{davide.giraudo1@univ-rouen.fr}

\date{\today}



 \begin{abstract}
  We provide some sufficient mixing conditions on a strictly stationary sequence 
  in order to guarantee the weak invariance principle in H\"{o}lder spaces. 
  Strong mixing, $\rho$-mixing  conditions are investigated as well as 
  $\tau$-dependent sequences. The main tools are 
  deviation inequalities for mixing sequences.
 \end{abstract}
 
 \maketitle

 \section{Introduction}
 
 \subsection{Context and notations}
 
 Let $(X_j)_{j\geqslant 0}$ be a strictly stationary sequence of real valued 
  random variables with zero mean and finite variance, and for an integer 
  $n\geqslant 1$, $S_n:=\sum_{j=1}^nX_j$ denotes the $n$-th 
  partial sum. 
  Its variance is denoted by $\sigma_n^2$. Let us consider the partial sum 
  process 
  \begin{equation}\label{polyg}
   S_n^{\mathrm{pl}}(t):=\sum_{j=1}^{[nt]}X_j+(nt-[nt])X_{[nt]+1},\quad 
   n\geqslant 1,t\in[0,1].
  \end{equation}
  We are interested in the asymptotic behavior of 
  $\sigma_n^{-1}S_n^{\mathrm{pl}}(\cdot)$ viewed as a 
  random function in some function spaces. 
  
  \begin{nota}
   If $T\colon \Omega\to\Omega$ is a bi-measurable measure preserving map, we 
   define 
   for $f\colon\Omega\to \R$ and a positive integer $n$ the $n$th partial sum 
   $S_n(f):=\sum_{j=1}^nf\circ T^j$ and $\sigma_n^2(f):=\mathbb E[S_n^2(f)]-(\mathbb 
   E[S_n(f)])^2$ denotes its variance. We shall also consider $S_n^{\mathrm{pl}}
   (f)$ defined in a similar way as in \eqref{polyg}, that is, 
   \begin{equation}\label{polyg_f}
   S_n^{\mathrm{pl}}(f,t):=S_{[nt]}(f)+ (nt-[nt])f\circ T^{[nt]+1},
   \end{equation}
   and $W_n(f,t):=n^{-1/2}S_n^{\mathrm{pl}}(f,t)$.
  \end{nota}

  In all the paper, the involved sequences of random variable are assumed 
  to be strictly stationary.
  
  When $(X_j)_{j\geqslant 0}$ is an independent identically distributed sequence, 
  Donsker showed (cf. \cite{MR0040613}) that 
  $(n^{-1/2}(\mathbb E(X_1^2))^{-1/2}S_n^{\mathrm{pl}})_{n\geqslant 1}$ 
  converges in distribution in the 
  space of continuous functions on the unit interval to a standard Brownian 
  motion $W$. 
  An intensive research has then been performed to extend this result to 
  stationary weakly dependent  sequences. 
  We refer the reader to \cite{MR2206313} for the main theorems in this area. 
  
  In this paper, we rather focus on the convergence in distribution of the 
  partial sum in other function spaces.

 \subsection{H\"older spaces}
 
 It is well known that standard Brownian motion's paths are almost surely 
  H\"older regular of exponent 
  $\alpha$ for each $\alpha\in (0,1/2)$, hence it is natural to consider the 
  random function defined in 
  \eqref{polyg_f} as an element of $\mathcal H_\alpha[0,1]$ and 
  try to establish its weak 
  convergence to a standard Brownian motion in this function space. 
  
  Before stating the results in this direction, let us 
  define for $\alpha\in (0,1)$ the H\"older space $\mathcal H_\alpha[0,1]$ of functions 
  $x\colon[0,1]\to\R$
  such that $\sup_{s\neq t}\abs{x(s)-x(t)}/\abs{s-t}^\alpha$ is finite. 
  The analogue of the continuity modulus in $C[0,1]$ is $w_\alpha$, defined by 
  \begin{equation*}
  w_\alpha(x,\delta)=\sup_{0<\abs{t-s}<\delta}\frac{\abs{x(t)-x(s)}}{\abs{t-
  s}^\alpha}.
  \end{equation*}
  We then define $\mathcal H_\alpha^0[0,1]$ by $\mathcal H_\alpha^0[0,1]:=
  \ens{x\in \mathcal H_\alpha[0,1],
  \lim_{\delta\to 
  0}w_\alpha(x,\delta)=0}$. We shall essentially work with the space 
  $\mathcal H_\alpha^0[0,1]$ which, 
  endowed with $\norm\cdot_\alpha\colon x\mapsto  w_\alpha(x,1)+\abs{x(0)}$, 
  is a separable Banach space 
  (while $\mathcal H_\alpha[0,1]$ is not). Since the canonical 
  embedding $\iota\colon \mathcal H^0_\alpha[0,1]\to \mathcal H_\alpha[0,1]$
  is continuous, each 
  convergence in distribution in $\mathcal H_\alpha^0[0,1]$ also takes place in 
  $\mathcal H_\alpha[0,1]$.

  In order to prove such a convergence, we need a tightness criterion. 
  Combining Theorem~14 in \cite{MR1736910} in the particular case 
  of the partial sum process \eqref{polyg_f} with Lemma~3.3  in  
  \cite{MR2759167}, the condition 
  \begin{equation}\label{tightness_criterion}
   \forall\varepsilon>0, \quad 
   \lim_{\delta \to 0}\limsup_{n\to \infty}\frac 1{\delta}\sum_{j=1}^{\log[n\delta]}2^j
  \mu\ens{\max_{1\leqslant k\leqslant 
  [n\delta]2^{-j}}\abs{S_k(f)}>\frac{\sigma_n}{n^\alpha}
  \varepsilon([n\delta]2^{-j})^\alpha}=0
  \end{equation}
  (where $\log$ denotes the binary logarithm)
  is sufficient for tightness of the sequence $(\sigma_n^{-1}(f)
  S_n^{\mathrm{pl}}(f))_{n\geqslant 1}$ in $\mathcal H_{\alpha}^0[0,1]$. 
  
  In the particular case of linear variance (that is, $\sigma_n^2\sim cn$ as $n\to \infty$
  for some constant $c$), we will normalize by $\sqrt n$. Using the change of indexes 
  $k=\log[n\delta]-j$ (so that $2^{-j}=2^k/[n\delta]$), 
  this leads to the following tightness criterion for 
  $(W_n(f))_{n\geqslant 1}$ in $\mathcal H_{\alpha}^0[0,1]$:
  
  \begin{equation}\label{tightness_criterion2}
   \forall\varepsilon>0, \quad 
   \lim_{\delta \to 0}\limsup_{n\to \infty}n\sum_{k=1}^{\log[n\delta]}2^{-k}
  \mu\ens{\max_{1\leqslant i\leqslant 
  2^{k}}\abs{S_i(f)}>
  \varepsilon 2^{k\alpha} n^{1/p}}=0,
  \end{equation}
  where $\alpha=1/2-1/p$. 
  
  As mentioned before, the random function defined in 
  \eqref{polyg} can be viewed 
  as an element of $\mathcal H_\alpha[0,1]$, $\alpha\in (0,1/2)$ and 
  we can try to establish the weak convergence of the sequence 
  $(\sigma_n^{-1}S_n^{\mathrm{pl}}(f))_{n\geqslant 1}$
  to a standard Brownian motion in this function space. 
  To the best of our knowledge, it seems that the study of this kind of 
  convergence was not as 
  intensive as in the space of continuous functions or the Skorohod space. 
  The first result in this direction was established by Lamperti in 
  \cite{MR0143245}: if 
  $(X_j)_{j\geqslant 0}$ is an i.i.d. sequence with $\mathbb E[X_0]=0$, $\mathbb 
  E[X_0^2]=1$ and 
  $\mathbb E\abs{X_0}^p$ is finite, then the sequence 
  $(n^{-1/2}S_n^{\mathrm{pl}})_{n\geqslant 1}$
  converges to a standard Brownian motion in $\mathcal H_{\gamma}^0[0,1]$ 
  for each $\gamma<1/2-1/p$. 
  Later, Ra\v{c}kauskas and Suquet improved this result (cf. \cite{MR2000642}), 
  showing that for an i.i.d. zero mean sequence, 
  a necessary and sufficient condition to obtain the invariance principle in 
  $\mathcal H^0_{1/2-1/p}[0,1]$ 
  is $\lim_{t\to\infty}t^p\mu\ens{\abs{X_0}>t}=0$ (in 
  \cite{MR2054586} they considered the case of more general H\"older spaces, where the role 
  of $t\mapsto t^\alpha$ is played by $t\mapsto t^\alpha L(t)$ with some 
  conditions on $L$). 
  
  Thus, establishing the weak convergence of the partial sum process in H\"older spaces 
  requires, even in the independent case, finite moment of order greater than $2$ and 
  the moment condition depends on the exponent of the considered H\"older space. It is 
  a natural question to ask about generalizations of the result by Ra\v{c}kauskas and 
  Suquet for dependent sequences. In this paper, we focus on strictly 
  stationary sequences satisfying some mixing conditions (see next section).

 \subsection{Mixing conditions}
  
  We present the mixing conditions involved in the paper.  
  
  Let $\mathcal A$ and $\mathcal B$ be two sub-$\sigma$-algebras of $\mathcal 
  F$, where $(\Omega,
  \mathcal F,\mu)$ is a probability space. We define the $\alpha$-mixing 
  coefficients as introduced by 
  Rosenblatt in \cite{MR0074711}:
  \begin{equation*}
  \alpha(\mathcal A,\mathcal B):=\sup\ens{\abs{\mu(A\cap B)-\mu(A)\mu(B)},A\in 
  \mathcal A,B\in \mathcal B}.
  \end{equation*}

  The $\rho$-mixing coefficients were introduced by Hirschfeld 
  \cite{PSP:1737020} and are defined by
  \begin{equation*}
  \rho(\mathcal A,\mathcal B):=\sup\ens{\abs{\operatorname{Corr}(f,g)},f\in 
  \mathbb L^2(\mathcal A),
  g\in  \mathbb L^2(\mathcal B),f\neq 0,g\neq 0},
  \end{equation*}
  where $\operatorname{Corr}(f,g):=\left[\mathbb E(fg)-\mathbb E(f)\mathbb 
  E(g)\right]\left[\norm{f-\mathbb 
  E(f)}_{\mathbb L^2}\norm{g-\mathbb E(g)}_{\mathbb L^2}\right]^{-1}$.

The coefficients are related by the inequalities
\begin{equation}\label{comp_mix}
4\alpha(\mathcal A,\mathcal B)\leqslant \rho(\mathcal A,\mathcal B).
\end{equation}

For a strictly stationary sequence $\suite{X_k}{k\in\Z}$ and $n\geqslant 0$ we 
define $\alpha_X(n) = 
\alpha(n) =
\alpha(\mathcal F_{-\infty}^0,\mathcal F_n^\infty)$
where $\mathcal F_u^v$ is the $\sigma$-algebra generated by $X_k$ with 
$u\leqslant k\leqslant v$ (if 
$u=-\infty$ or $v=\infty$, the corresponding
inequality is strict). In the same way we define coefficients $\rho_X(n)$.

When there will be no ambiguity, we shall simply write $\alpha(n)$ and  
$\rho(n)$. We say that the sequence $\suite{X_k}{k\in\Z}$ is $\alpha$\textit{-mixing}
if $\ds\lim_{n\to +\infty}\alpha(n)=0$, and similarily we define 
$\rho$-mixing sequences. 

$\alpha$-mixing sequences were considered in the mentioned references, 
while $\rho$-mixing sequences first appeared in \cite{MR0133175}.
 Inequality \eqref{comp_mix} translated in terms of mixing coefficients of a 
 sequence states that for each positive integer $n$,  
 \begin{equation*}
  4\alpha(n)\leqslant\rho(n).
 \end{equation*}
In particular, a $\rho$-mixing sequence is $\alpha$-mixing.

\subsection{$\tau$-dependence coefficient}

In order to define the $\tau$-dependence coefficients of a stationary sequence, 
we first need a result about conditional probability (see Theorem 33.3 of 
\cite{MR1324786}).
\begin{lem}
Let $(\Omega,\mathcal F,\mu)$ be a probability space, $\mathcal M$ a 
sub-$\sigma$-algebra of $\mathcal F$ and $X$ a real-valued random variable 
with distribution $\mu_X$. There exists a function $\mu_{X\mid \mathcal M}$ 
from $\mathcal B(\R)\times\Omega$ to $[0,1]$ such that 
\begin{enumerate}
 \item For any $\omega\in \Omega$, $\mu_{X\mid\mathcal M}(\cdot,\omega)$ is 
 a probability measure on $\mathcal B(\R)$.
 \item For any $A\in\mathcal B(\R)$, $\mu_{X\mid\mathcal M}(A,\cdot)$ is a 
 version of $\mathbb E[\mathbf 1_{\ens{X\in A}}\mid\mathcal M]$.
\end{enumerate} 
\end{lem} 
We now introduce the $\tau$-dependence coefficients as in \cite{MR2199291}.  
We denote by $\Lambda_1(\R)$ the collection of $1$-Lipschitz functions from 
$\R$ to $\R$ and define the quantity
\[W(\mu_{X\mid\mathcal M}):=\sup\ens{
\abs{\int f(x)\mu_{X\mid\mathcal M}(\mathrm dx)-
\int f(x)\mu_X(\mathrm dx)},f\in\Lambda_1(\R)}.\]
For an integrable random variable $X$ and a sub-$\sigma$-algebra $\mathcal M$, 
the coefficient $\tau$ is defined by 
\begin{equation}
\tau(\mathcal M,X)=\norm{W(\mu_{X\mid\mathcal M})}_1.
\end{equation}
This definition can be extended to random variables with values in 
finite dimensional vector spaces. If $d$ is a positive integer, we endow 
$\R^d$ with the norm $\lVert x-y\rVert:=\sum_{j=1}^n\abs{x_j-y_j}$ and 
define $\Lambda_1(\R^d)$ as the set of $1$-Lipschitz functions from $\R^d$ 
to $\R$. 
\begin{dfn}
Let $(\Omega,\mathcal F,\mu)$ be a probability space, $\mathcal M$ a 
sub-$\sigma$-algebra of $\mathcal F$ and $X$ an $\R^d$-valued random variable. 
We define 
\begin{equation}
\tau(\mathcal M,X):=\sup\ens{\tau(\mathcal M,f(X)),f\in 
\Lambda_1(E)}.
\end{equation}
\end{dfn}
We can now introduce the $\tau$-mixing coefficient for a sequence of real-valued 
random variables.
\begin{dfn}
Let $(X_i)_{i\geqslant 1}$ be a sequence of random variables and $(\mathcal 
M_i)_{i\geqslant 1}$ a sequence of sub-$\sigma$-algebras of $\mathcal F$. For 
any positive integer $k$, define 
\begin{equation}
\tau(i):=\max_{p,l\geqslant 1}\frac 1l\sup\ens{
\tau(\mathcal M_p,(X_{j_1},\dots,X_{j_l})), p+i\leqslant j_1<\dots<j_l}. 
\end{equation}
\end{dfn}
In the sequel, we shall focus on the case $\mathcal M_i:=\sigma(X_k,
k\leqslant i)$. 
\begin{nota}\label{notation_inverse}
 Let $X\colon\Omega\to\R$ be a random variable. We denote $Q_X(\cdot)$ 
 the inverse function defined by 
 $Q_X(u):=\inf\ens{t,\mu\ens{\abs X>t}\leqslant u}$. If  $(f\circ T^j )_{j
 \geqslant 0}$ is a strictly stationary sequence and 
 $(\alpha(n))_{n\geqslant 1}$ is its sequence of 
 $\alpha$-mixing coefficients, we denote by $\alpha^{-1}(u)$ the number of indices $n$ 
 for which $\alpha(n)\geqslant u$. More generally, if $(\delta_i)_{i\geqslant 
 0}$ is a non-increasing sequence of non-negative numbers, we define 
 $\delta^{-1}(u):=\inf\ens{k\in \N,\delta_k\leqslant u}$.
 \end{nota}
 
We can compare the $\tau$-dependence coefficient with the $\alpha$-mixing 
coefficients. The following is a simplified version of Lemma~7 of 
\cite{MR2105738}.
\begin{lem}\label{Lemma_comparison_tau_alpha}
Let $(f\circ T^j)_{j\geqslant 0}$ be a strictly stationary sequence. Then for each 
integer $i$, the following inequality holds:
\begin{equation}\label{comparison_tau_alpha}
\tau(i)\leqslant 2\int_0^{2\alpha(i)}Q_{f}(u)\mathrm du.
\end{equation}
\end{lem} 
In \cite{MR2105738}, "Application~1: causal linear processes" (p. 871), 
Dedecker and Prieur provide an example of a process whose 
$\tau$-dependence coefficients converge to $0$ as fast as $2^{-i}$ but 
$\alpha_i=1/4$ for each positive integer $i$.

 \section{Main results}    
  
  \subsection{Mixing sequences}
  
  In this subsection, we give sufficient mixing conditions which guarantee 
  the convergence of the sequence $(W_n(f))_{n\geqslant 1}$ to 
  a Brownian motion in the space $\mathcal H^0_{1/2-1/p}[0,1]$, $p>2$.
  
  We refer the reader to Notations~\ref{notation_inverse} and~\ref{notation_G}.
  \begin{theo}\label{theorem_tau_dependent}
    Let $p>2$ and let
    $(f\circ T^j)_{j\geqslant 0}$ be a strictly stationary centered sequence such that 
   \begin{equation}\label{condition_tau_dependence}
    \lim_{t\to \infty}t^{p-1}\int_0^1Q_f(u)
    \mathbf 1\ens{\left(\left(\tau/2\right)^{-1}\circ G_{f}^{-1}
  \right)(u)Q_f(u)>t}\mathrm du=0.
   \end{equation}
  Then  
  \begin{equation}\label{convergence_tau}
  W_n(f)\to \sigma(f) W
  \mbox{ in distribution in }\mathcal H_{1/2-1/p}^0[0,1],
  \end{equation}
  where $\sigma^2(f)=\operatorname{Var}(f)+2\sum_{k=1}^\infty
  \operatorname{Cov}(f,f\circ T^k)$.
  \end{theo}
 Using the comparison between $\alpha$ and $\tau$, we can deduce a condition in 
 the spirit of that of Doukhan, Massart and Rio (see \cite{MR1324814}). One 
 can also derive it from the tightness criterion \eqref{tightness_criterion2} and 
 Theorem 6.2 of \cite{MR2117923}.
  \begin{cor}\label{corollary_alpha_mixing}
    Let $p>2$ and let $(f\circ T^j)_{j\geqslant 0}$ be a strictly 
    stationary centered sequence such that 
   \begin{equation}\label{condition_alpha}
    \lim_{t\to \infty}t^{p-1}\int_0^1Q_f(u)
    \mathbf 1\ens{\alpha^{-1}(u)Q_{f}(u)>t}\mathrm du=0.
   \end{equation}
  Then  
  \begin{equation}\label{convergence_alpha}
  W_n(f)\to \sigma(f) W
  \mbox{ in distribution in }\mathcal H_{1/2-1/p}^0[0,1],
  \end{equation}
  where $\sigma^2(f)=\lim_{n\to\infty}\sigma_n^2(f)/n$.
  \end{cor}

  \begin{rem}\label{iid_case}
   Assume that the sequence $(f\circ T^j)_{j\geqslant 0}$ is 
   independent and that $t^p\mu\ens{\abs f>t}\to 0$. Then 
   the condition of Theorem~\ref{theorem_tau_dependent} is 
   satisfied. Indeed, since $Q_f(U)$ is distributed as $\abs f$ if 
   $U$ is uniformly distributed on $[0,1]$, both 
   conditions \eqref{convergence_tau} and 
   \eqref{convergence_alpha} read 
   \begin{equation}\label{condition_independent}
     \lim_{t\to \infty}t^{p-1}\mathbb E\left[\abs f\mathbf 1\ens{\abs f>t}\right]=0.
   \end{equation}
  As 
  \begin{align*}
   t^{p-1}\mathbb E\left[\abs f\mathbf 1\ens{\abs f>t}\right]&=
   t^{p-1}\int_0^\infty\mu\ens{\abs f>\max\ens{u,t}}\mathrm du\\
   &=t^{p}\mu\ens{\abs f>t}+t^{p-1}\int_t^\infty\mu\ens{\abs f
   >u}\mathrm du\\
   &\leqslant t^{p}\mu\ens{\abs f>t}+\sup_{s\geqslant t}s^p
   \mu\ens{\abs f>s}/(p-1),
  \end{align*}
  condition \eqref{condition_independent} is satisfied hence we can derive 
  the result by Ra\v{c}kauskas and Suquet in the i.i.d. case from 
  Theorem~\ref{theorem_tau_dependent}.This contrasts with Theorem~17 
  of \cite{MR1759810}, from which we can 
  only deduce the result by Lamperti (cf. \cite{MR0143245}) in the 
  i.i.d. case.
  \end{rem}

  \begin{rem}
   Assume that $Q_f(u)\leqslant Cu^{-1/a}$ for some $a>p$ (this is the case 
   if $f$ admits a finite weak moment of order $a$). If $\alpha(k)=o(k^{-a(p-1)/(a-p)})$ or 
   $\tau(k)=o(k^{-(a-1)(p-1)/(a-p)})$, then condition \eqref{condition_tau_dependence} 
   holds. If $f$ is bounded, these sufficient conditions can be weakened respectively to 
   $\alpha(k)=o(k^{-(p-1)})$ and $\tau(k)=o(k^{-(p-1)})$.
  \end{rem}

  We conclude this subsection by a result on $\rho$-mixing sequences.
  
  \begin{theo}\label{theorem_rho_mixing}
   Let $p>2$ and let $(f\circ T^j)_{j\geqslant 0}$ be a strictly stationary
   centered sequence such that 
   $t^p\mu\ens{\abs f>t}\to 0$ as $t\to +\infty$ and $\sum_{i=0}^{\infty}\rho(2^i)<\infty$. Then  
  \begin{equation}\label{convergence_rho}
  W_n(f)\to \sigma(f) W
  \mbox{ in distribution in }\mathcal H_{1/2-1/p}^0[0,1],
  \end{equation}
  where $\sigma^2(f)=\lim_{n\to\infty}\sigma_n^2(f)/n$.
  \end{theo}

  \subsection{A counter-example}
  
  In this subsection, we show that boundedness of the sequence of $p$th moments
  of the normalized partial sums is not enough to guarantee tightness in $\mathcal H_{1/2-1/p}[0,1]$.

  \begin{theo}\label{counter_example}
   Let $p>2$. There exists a strictly stationary sequence $(f\circ T^j)_{j\geqslant 0}$ 
   such that 
   \begin{itemize}
    \item the finite dimensional distributions of $(W_n(f))_{n\geqslant 1}$ 
    converge to those of a standard Brownian motion,
    \item the sequence $(\mathbb E\abs{S_n(f)}^p/n^{p/2})_{n\geqslant 1}$  is bounded and 
    \item the process $(W_n(f))_{n\geqslant 1}$ is not tight in 
    $\mathcal H_{1/2-1/p}[0,1]$.
   \end{itemize}
  \end{theo}
  
   \begin{rem}\label{other_methods}
 The constructed process has no reason to be $\alpha$-mixing. However, this 
 proves that in general, establishing tightness in $\mathcal H_{1/2-1/p}^0[0,1]$ 
 of $(W_n(f))_{n\geqslant 1}$
 cannot be done by proving boundedness in $\mathbb L^p$ of the sequence 
 $(W_n(f))_{n\geqslant 1}$. Thus other methods need to be used.
 \end{rem}

Let us recall that a sequence $(c_n)_{n\geqslant 1}$ is \textit{slowly varying} if there exists 
   a continuous function $h\colon \R_+^*\to\R_+^*$ such that $c_n=h(n)$ for each 
   positive integer $n$ and for each positive $x$, $\lim_{t\to \infty}h(tx)/h(t)=1$. 
 \begin{rem}\label{remark_moment_bound}
  If $p>2$ and $(f\circ T^j)_{j\geqslant 0}$ is a strictly stationary 
  centered sequence such that the finite dimensional distributions 
  of $(\sigma_n^{-1}S_n^{\mathrm{pl}}(f))_{n\geqslant 1}$ 
  converge to those of a standard Brownian motion, the sequence $(\sigma_n^2(f)/n)_{n\geqslant 1}$ 
  is slowly varying, and the sequence $\left(\mathbb E\abs{S_n(f)}^p/\sigma_n^p\right)_{n\geqslant 1}$ 
  is bounded, then for each $\gamma<1/2-1/p$ the sequence 
  $(\sigma_n^{-1}S_n^{\mathrm{pl}}(f))_{n\geqslant 1}$ 
  converges in distribution in $\mathcal H_\alpha^0[0,1]$ to a standard Brownian motion.
  This can be seen using tightness criterion \eqref{tightness_criterion}, 
  Markov's inequality and 
  boundedness in $\mathbb L^p$ of $\left(\sigma_n^{-1}\max_{1\leqslant j
  \leqslant n}\abs{S_j(f)}\right)_{n\geqslant 1}$ (by Serfling
  arguments, see \cite{MR0268938}).
 \end{rem}
  
 \section{Proofs}  

\begin{proof}[Proof of Theorem~\ref{theorem_tau_dependent}]
Notice that \eqref{condition_tau_dependence} implies finiteness of
$\int_0^1Q_f^2(u)(\tau/2)^{-1}\circ G_f^{-1}(u)\mathrm du$, hence
condition~(5.5) in \cite{MR2105738}. This implies the 
convergence of $\left(\sigma_n^2(f)/n\right)_{n\geqslant 1}$ to 
$\sigma(f)$. Since $\theta(k)$ is smaller than 
$\tau(k)$, Corollary~1 in \cite{MR1983043} shows that 
the function $f$ satisfies the projective criterion by Dedecker and Rio 
(see \cite{MR1743095}), from which the convergence of the 
finite dimensional distributions follows. It remains to check tightness of 
$(W_n(f))_{n\geqslant 1}$ in 
$\mathcal H_{1/2-1/p}^0[0,1]$. We shall check that \eqref{tightness_criterion2} 
is satisfied. To this aim, we apply Theorem~\ref{Theorem_tail_inequality_tau_dep} for 
each $k\in \ens{1,\dots,\log [n\delta]}$ 
with some $r>p$, $N:=2^k$ and $\lambda:=\varepsilon 2^{k\alpha}n^{1/p}$ . This gives 
\begin{multline*}
 n\sum_{k=1}^{\log[n\delta]}2^{-k}
  \mu\ens{\max_{1\leqslant i\leqslant 
  2^{k}}\abs{S_i(f)}>
  5\varepsilon 2^{k\alpha} n^{1/p}}
  \leqslant n\sum_{k=1}^{\log[n\delta]}2^{-k}
  4r^{r/2}s_{2^k}(f)^r\left(\varepsilon 2^{k\alpha}n^{1/p}\right)^{-r}
  +\\
  +n\sum_{k=1}^{\log[n\delta]}(\varepsilon 2^{k\alpha}n^{1/p})^{-1}
  \int_0^1Q(u)\mathbf 1\ens{R(u) \geqslant \varepsilon 2^{k\alpha}n^{1/p}/r}\mathrm du.
\end{multline*}
By \eqref{condition_tau_dependence}, the quantity $C:=
\int_0^{\norm{f}_1}(\tau/2)^{-1}(u)
  Q_{f}\circ G_{f}(u)\mathrm du$ is finite.
In view of \eqref{control_variance_tau}, we thus have 
\begin{multline}\label{tightness_tau1}
 n\sum_{k=1}^{\log[n\delta]}2^{-k}
  \mu\ens{\max_{1\leqslant i\leqslant 
  2^{k}}\abs{S_i(f)}>
  5\varepsilon 2^{k\alpha} n^{1/p}}
  \leqslant 4\cdot (4r)^{r/2}C^{r/2}n\sum_{k=1}^{\log[n\delta]}2^{-k}
  2^{kr/2}\left(\varepsilon 2^{k\alpha}n^{1/p}\right)^{-r}
  +\\
  +n^{1-1/p}\sum_{k=1}^{\log[n\delta]}(\varepsilon 2^{k\alpha})^{-1}
  \int_0^1Q(u)\mathbf 1\ens{R(u) \geqslant \varepsilon 2^{k\alpha}n^{1/p}/r}\mathrm du=:(I)+
  (II).
\end{multline}
A simple computation shows that 
\begin{equation}\label{tightness_tau2}
 (I)\leqslant K(r,p,\varepsilon)\delta^{r/p-1},
\end{equation}
and for the second term, we have the upper bound 
\begin{equation}\label{tightness_tau3}
 (II)\leqslant K(\alpha,\varepsilon)
 n^{(p-1)/p}
 \int_0^1Q(u)\mathbf 1\ens{(\tau/2)^{-1}\circ G_f^{-1}(u)Q(u) \geqslant \varepsilon n^{1/p}/r}\mathrm du.
\end{equation}
Since $r>p$, the condition \eqref{tightness_criterion2} is satisfied 
in view of \eqref{tightness_tau1}, \eqref{tightness_tau2}, 
\eqref{tightness_tau3} and \eqref{condition_tau_dependence}.

\end{proof}

\begin{proof}[Proof of Corollary~\ref{corollary_alpha_mixing}]
 It suffices to check that condition \eqref{condition_alpha} implies 
 \eqref{condition_tau_dependence}. Notice that by 
 \eqref{comparison_tau_alpha}, we have for a fixed $v$, 
 \begin{equation}
  \inf\ens{i\mid \tau(i)/2\leqslant v}\leqslant 
  \inf\ens{i\mid G^{-1}(2(\alpha(i))) \leqslant v}=
  \inf\ens{i\mid \alpha(i)\leqslant G(v)/2}
 \end{equation}
 hence $(\tau/2)^{-1}(v)\leqslant \alpha^{-1}(G(v)/2)$. Taking 
 $v=G^{-1}(u)$ for a fixed $u$, we get 
 \begin{equation}
  (\tau/2)^{-1}\circ G^{-1}(u)\leqslant \alpha^{-1}(u/2).
 \end{equation}
 Since the function $u\mapsto \alpha^{-1}(u)$ is non-increasing, 
 the inclusion
 \begin{equation*}
\ens{\left(\tau/2\right)^{-1}\circ G^{-1}(u)Q_f(u)> t}
\subset \ens{\alpha^{-1}(u/2)Q_f(u/2)> t}.
 \end{equation*}
 takes place. As a consequence, we obtain 
 \begin{multline}
 t^{p-1}\int_0^1Q_f(u)
    \mathbf 1\ens{\left(\left(\tau/2\right)^{-1}\circ G_{f}^{-1}
  \right)(u)Q_f(u)>t}\mathrm du\leqslant \\
  \leqslant
  t^{p-1}\int_0^1Q_f(u/2)\mathbf 1\ens{\alpha^{-1}(u/2)Q_f(u/2)> t}\mathrm du,
 \end{multline}
 which concludes the proof of Corollary~\ref{corollary_alpha_mixing}.
\end{proof}

\begin{proof}[Proof of Theorem~\ref{theorem_rho_mixing}]
 Theorem~4.1. of \cite{MR672297} guarantees the existence of the limit 
of the sequence $(\sigma_n^2/n)_{n\geqslant 1}$ and \cite{MR996917} gives the 
convergence of the finite dimensional distributions. Therefore, the proof will 
be finished if we check the convergence \eqref{tightness_criterion2}. 
We apply for Theorem~\ref{tail_inequality_Shao} for each 
$1\leqslant k\leqslant \log[n\delta]$ with a $q>p$, $N:=2^k$ $x:=\varepsilon 2^{k\alpha}n^{1/p}$
and $A:=2^{k\alpha}n^{1/p}\eta$, where $\eta$ is fixed (notice that since 
\begin{equation}
 \mathbb E\left[\abs f \mathbf 1\ens{\abs f\geqslant A}\right]
 =A\mu\ens{\abs f\geqslant A}+\int_A^{+\infty}\mu\ens{\abs f\geqslant t}\mathrm dt
 \leqslant C(p,f)A^{1-p},
\end{equation}
we have for $n\geqslant n(\eta,p,\varepsilon,f)$ and $1\leqslant k\leqslant \log [n\delta]$,
\begin{equation}
 2\cdot 2^{k}\cdot \mathbb E\left[\abs f \mathbf 1\ens{\abs f\geqslant A}\right]
 \leqslant 2C(p,f)(\eta 2^{k\alpha}n^{1/p})^{1-p}
 \leqslant \varepsilon 2^{k\alpha}n^{1/p}=x,
\end{equation}
hence \eqref{condition_on_A} is satisfied). This yields 
\begin{multline}\label{bound_rho_mixing1}
 n\sum_{k=1}^{\log[n\delta]}2^{-k}
  \mu\ens{\max_{1\leqslant i\leqslant 
  2^{k}}\abs{S_i(f)}\geqslant
  \varepsilon 2^{k\alpha} n^{1/p}}\leqslant n\sum_{k=1}^{\log[n\delta]}
  \mu\ens{\abs f\geqslant2^{k\alpha}n^{1/p}}+\\
  +K\exp\left(K\sum_{i=0}^\infty\rho(2^i)\right)
  n\sum_{k=1}^{\log[n\delta]}2^{-k}2^{kq/2}(\varepsilon 2^{k\alpha}n^{1/p})^{-q}
 \norm{f}_2^ q  +\\
 +K n\sum_{k=1}^{\log[n\delta]}\exp\left(K\sum_{i=0}^k\rho^{2/q}(2^i)\right)
 (\varepsilon 2^{k\alpha}n^{1/p})^{-q}\norm{f\mathbf 1\ens{\abs f\leqslant 
 \eta 2^{k\alpha}n^{1/p}}}_q^q.
\end{multline}
Since for some constant $C$ depending only on $f$ and $p$, the bound
\begin{equation}
 \norm{f\mathbf 1\ens{\abs f\leqslant \eta 2^{k\alpha}n^{1/p}}}_q^q
 \leqslant C(\eta é2^{k\alpha}n^{1/p})^{q-p}
\end{equation}
is valid, we derive from \eqref{bound_rho_mixing1} the inequality 
\begin{multline*}
 n\sum_{k=1}^{\log[n\delta]}2^{-k}
  \mu\ens{\max_{1\leqslant i\leqslant 
  2^{k}}\abs{S_i(f)}\geqslant
  \varepsilon 2^{k\alpha} n^{1/p}}\leqslant 
  \varepsilon^{-p}\eta^{-p}\sup_{t\geqslant n^{1/p}}t^p\mu\ens{\abs f\geqslant t}
  \sum_{k=1}^{\infty}2^{-kp\alpha}
  +\\
  +K\exp\left(K\sum_{i=0}^\infty\rho(2^i)\right)
  \delta^{q/p-1}\cdot \frac 1{2^{q/p}-1}\varepsilon^{-q}
 \norm{f}_2^ q  +\\
 +KC\varepsilon^{-q}\eta^{q-p} \sum_{k=1}^{\log[n\delta]}\exp\left(K\sum_{i=0}^k\rho^{2/q}(2^i)\right)
 2^{-kp\alpha}.
\end{multline*}
Since $\lim_{t\to \infty}t^p\mu\ens{\abs f>t}=0$, we obtain for each $\delta$ and $\eta$
\begin{multline*}
 \limsup_{n\to \infty}
  n\sum_{k=1}^{\log[n\delta]}2^{-k}
  \mu\ens{\max_{1\leqslant i\leqslant 
  2^{k}}\abs{S_i(f)}>
  \varepsilon 2^{k\alpha} n^{1/p}}\leqslant
  K\exp\left(K\sum_{i=0}^\infty\rho(2^i)\right)
  (2\delta)^{q/p-1}\varepsilon^{-q}
 \norm{f}_2^ q  +\\
 +KC\varepsilon^{-q}\eta^{q-p} \sum_{k=1}^{\infty}\exp\left(K\sum_{i=0}^k\rho^{2/q}(2^i)\right)
 2^{-kp\alpha},
\end{multline*}
from which \eqref{tightness_criterion2} follows (the convergence of the last series 
is ensured by the ratio test and the convergence to $0$ of $\rho^{2/q}(2^k)$). 
\end{proof}

  \begin{proof}[Proof of Theorem~\ref{counter_example}]
  We assume that $(\Omega,\mathcal F,\mu,T)$ is a non-atomic invertible
  measure preserving system.
  We shall first construct a function $g$ such that:
  \begin{enumerate}
   \item the sequence $(\mathbb E\abs{S_n(g-g\circ T)}^p/n^{p/2}
   )_{n\geqslant 1}$ is bounded and 
   \item the process $(W_n(g-g\circ T)
   )_{n\geqslant 1}$ is not tight in $\mathcal H_{1/2-1/p}^0[0,1]$.
  \end{enumerate}
  We then consider $f:=m+g-g\circ T$, where $m$ is such that 
  $(m\circ T^j)_{j\geqslant 0}$ is a martingale difference 
  sequence with $m\in\mathbb L^p$ and $m\neq 0$. This 
  will guarantee the convergence of the finite dimensional 
  distributions of $(W_n(f))_{n\geqslant 1}$ 
  to those of a scalar multiple of a standard Brownian motion, and 
  Burkholder's inequality ensures boundedness of the sequence 
  $(\mathbb E\abs{S_n(f)}^p/n^{p/2}
   )_{n\geqslant 1}$. We use a construction similar to 
   that given in \cite{MR1893125}. Let us consider two increasing 
   sequences of integer $(K_l)_{l\geqslant 1}$ and $(N_l)_{l\geqslant 1}$ 
   satisfying for each $l\geqslant 2$:
   \begin{align}
   \lim_{l\to +\infty}N_l\sum_{l'>l}K_{l'}/N_{l'}=0;\label{lac1}\\
  4N_{l}^{-1/p}\cdot l\cdot N_{l-1}<1;  \label{lac2}\\
  \sum_{i=1}^lK_i^{1/2}\leqslant K_{l+1}^{1/2};\label{lac3}\\
  \sum_{l=1}^{+\infty}\frac{K_l}{K_{l+1}^{1/2}}<\infty\label{lac4}.
 \end{align}
  We also assume that $4K_l\leqslant N_l$ for each $l$. 
  
  Let us fix an integer $l$. Using Rokhlin's lemma, we can find a 
  set $A_l\in\mathcal F$ such that the set $T^iA_l$, $0\leqslant i\leqslant N_l-1$
   are pairwise disjoint and $\mu\left(\bigcup_{i=0}^{N_l-1}T^iA_l\right)
  \geqslant 1/2$. We define 
 \[h_l:=\frac{N_l^{1/p}}{K_l^{1/2+1/p}}\cdot\mathbf 1\left(\bigcup_{j=1}^{K_l}T^{N_l-j}A_l\right);\]
 \[g_l:=\sum_{j=0}^{K_l-1}h_l\circ T^j=\frac{N_l^{1/p}}{K_l^{1/2+1/p}}\left(\sum_{j=1}^{K_l}j
 \mathbf 1(T^{N_l-j}A_l)+
 \sum_{j=K_l+1}^{2K_l-1}(2K_l-j)\mathbf 1(T^{N_l-j}A_l)\right);\]
 \[g:=\sum_{l=1}^{+\infty}g_l.\]

Assume that $\omega\in A_l$ and $N_l-K_l\leqslant i\leqslant N_l-1$. Then 
$g_l\circ T^i(\omega)=N_l-i$. Consequently, for $i,i'\in\ens{N_k-K_l,\dots, 
N_l-1}$,  
\[
 \abs{g_l\circ T^i-g_l\circ T^{i'}}\geqslant \frac{N_l^{1/p}}{K_l^{1/2+1/p}}
 \abs{i'-i}\mathbf 1(A_l).
\]
Applying $U^{-k}$ on both sides of the previous inequality 
for $0\leqslant k\leqslant N_l-K_l$ and taking the 
maximum over these $k$, we obtain 
\[
 \max_{0\leqslant k\leqslant N_l-K_l}
 \abs{g_l\circ T^{i-k}-g_l\circ T^{i'-k}}\geqslant \abs{i'-i}
 \frac{N_l^{1/p}}{K_l^{1/2+1/p}}\mathbf 1\left(\bigcup_{k=0}^{N_k-K_l}
 T^k(A_l)\right).
\]
This implies 
\begin{equation}\label{borne_inf_g_l}
 \frac 1{N_l^{1/p}}\max_{1\leqslant i<i'\leqslant N_l}
 \frac{\abs{g_l\circ T^{i'}-g_l\circ T^i}}{(i'-i)^{1/2-1/p}}\geqslant 
 \mathbf 1\left(\bigcup_{k=0}^{N_k-K_l}
 T^k(A_l)\right).
\end{equation}

 For $l'<l$, noticing that $\abs{g_{l'}}\leqslant N_{l'}^{1/p}K_{l'}^{1/2-1/p}$, the following inequality 
 takes place:
 \begin{equation*}
 \max_{1\leqslant i<j\leqslant N_l}\abs{g_{l'}\circ T^j-g_{l'}\circ T^i}\\
 \leqslant 2N_{l'}^{1/p}K_{l'}^{1-1/p},
 \end{equation*}
 therefore, 
 \begin{multline*}
 \frac 1{N_l^{1/p}}\max_{1\leqslant i<i'\leqslant N_l}\frac{\abs{\sum_{l'<l}g_{l'}\circ T^{i'}
 -\sum_{l'<l}g_{l'}\circ T^{i}}}
  {(i'-i)^{1/2-1/p}}\leqslant \\
  \leqslant 
  \frac 1{N_l^{1/p}}\sum_{l'<l}2N_{l'}^{1/p}K_{l'}^{1-1/p}
  \leqslant 2N_l^{-1/p}lN_{l-1}.
 \end{multline*}
 By condition~\eqref{lac2}, we conclude that 
 \begin{equation}\label{borne}
  \frac 1{N_l^{1/p}}\max_{1\leqslant i<i'\leqslant N_l}\frac{\abs{\sum_{l'<l}g_{l'}\circ T^{i'}
 -\sum_{l'<l}g_{l'}\circ T^{i}}}
  {(i'-i)^{1/2-1/p}}\leqslant 1/2.
 \end{equation}
 Moreover, notice that 
  \begin{multline*}
  \mu\ens{N_l^{-1/p}\max_{1\leqslant i<i'\leqslant N_l}\frac{\abs{\sum_{l'>l}g_{l'}\circ T^{i'}
 -\sum_{l'>l}g_{l'}\circ T^{i}}}
  {(i'-i)^{1/2-1/p}}\neq 0}\\
  \leqslant \sum_{l'>l}\mu\ens{\max_{1\leqslant i<i'\leqslant N_l}\abs{g_{l'}\circ T^{i'}
 -g_{l'}\circ T^{i}}\neq 0}\\
  \leqslant N_l\sum_{l'>l}\mu\ens{g_{l'}\neq 0},
 \end{multline*}
 hence 
 \begin{multline}\label{support}
  \mu\ens{\frac 1{N_l^{1/p}}\max_{1\leqslant i<i'\leqslant N_l}\frac{\abs{\sum_{l'>l}g_{l'}\circ T^{i'}
 -\sum_{l'>l}g_{l'}\circ T^{i}}}
  {(i'-i)^{1/2-1/p}}\neq 0}\leqslant\\
  \leqslant 2N_l\sum_{l'>l}K_{l'}/N_{l'}.
 \end{multline}
  By \eqref{borne} and \eqref{support}, we get 
 \begin{multline*}
  \mu\ens{\max_{1\leqslant i<i'\leqslant N_l}\frac{\abs{g\circ T^{i'}-g\circ T^{i}}}
  {(i'-i)^{1/2-1/p}}\geqslant N_l^{1/p}/2}\geqslant\\ 
  \geqslant 
  \mu\ens{\max_{1\leqslant i<i'\leqslant N_l}\frac{\abs{\sum_{l'\geqslant l}(g_{l'}\circ T^{i'}-g_{l'}\circ T^{i})}}
  {(i'-i)^{1/2-1/p}}\geqslant N_l^{1/p}}\\
  \geqslant \mu\ens{\max_{1\leqslant i<i'\leqslant N_l}\frac{\abs{(g_{l}\circ T^{i'}-g_{l}\circ T^{i})}}
  {(i'-i)^{1/2-1/p}}\geqslant N_l^{1/p}}
  -2N_l\sum_{l'>l}K_{l'}/N_{l'}.
 \end{multline*}
 Combining the previous inequality with \eqref{borne_inf_g_l}, we obtain for each 
 integer $l$, 
 \begin{equation}\label{ineg_non_tens}
  \mu\ens{\max_{1\leqslant i<i'\leqslant N_l}\frac{\abs{g\circ T^{i'}-g\circ T^{i}}}
  {(i'-i)^{1/2-1/p}}\geqslant N_l^{1/p}/2}\geqslant \frac 12-\frac{K_l}{2N_l}-2N_l\sum_{l'>l}K_{l'}/N_{l'},
 \end{equation}
 and by \eqref{lac1}, the inequality 
 \begin{equation}
  \mu\ens{\max_{1\leqslant i<i'\leqslant N_l}\frac{\abs{g\circ T^{i'}-g\circ T^{i}}}
  {(i'-i)^{1/2-1/p}}\geqslant \frac{N_l^{1/p}}2}\geqslant \frac 18
 \end{equation}
 holds for $l$ large enough. We deduce that for such integers $l$ and each 
 $\delta\in (0,1)$, 
 \begin{equation}\label{ineg_non_tens2}
  \mu\ens{w_{1/2-1/p}\left(\frac 1{
   \sqrt{N_l}}S_{N_l}^{\mathrm{pl}}(g-g\circ T),\delta\right)\geqslant 1/2}\geqslant \frac 18,
 \end{equation}
 hence the process $(W_n(g-g\circ T))_{n\geqslant 1}$ cannot be tight 
 in $\mathcal H_{1/2-1/p}[0,1]$.
 
 It remains to show that the sequence $(n^{-1/2}(g-g\circ T^n))_{n\geqslant 1}$ is 
 bounded in $\mathbb L^p$. Notice that for a fixed integer $l\geqslant 1$, the 
 equalities 
 \[
 \abs{g_l-g_l\circ T}=\abs{h_l-h_l\circ T^{K_l}}=\frac{N_l^{1/p}}{K_l^{1/2+1/p}}\cdot 
 \mathbf 1\left(\bigcup_{j=1}^{2K_l}T^{N_l-j}A_l\right)
 \]
 take place. This implies that 
 \begin{align*}
  \norm{g_l-g_l\circ T}_p&=\frac{N_l^{1/p}}{K_l^{1/2+1/p}}\cdot 
  \left[\mu\left(\bigcup_{j=1}^{2K_l}T^{N_l-j}A_l\right)\right]^{1/p}\\
  &\leqslant \frac{N_l^{1/p}}{K_l^{1/2+1/p}}\left(\frac{2K_l}{N_l}\right)^{1/p}=2^{1/p}K_l^{-1/2},
 \end{align*}
 hence for each integer $n\geqslant 1$, $\norm{g_l-g_l\circ T^n}_p\leqslant 2^{1/p}nK_l^{-1/2}$. 
 Let us define 
 \[
  \widetilde{g_l}:=\frac{N_l^{1/p}}{K_l^{1/2+1/p}}\sum_{j=1}^{K_l}j\mathbf 1(T^{N_l-j}A_l).
 \]
 
 If $K_l\leqslant n$, then 
 \begin{align*}
 \widetilde{g_l}-\widetilde{g_l}\circ T^n&=\frac{N_l^{1/p}}{K_l^{1/2+1/p}}\left(\sum_{j=1}^{K_l} j 
 \mathbf 1(T^{N_l-j}A_l)-\sum_{j=1}^{K_l} j \mathbf 1(T^{N_l-j-n}A_l)\right)\\
 &=\frac{N_l^{1/p}}{K_l^{1/2+1/p}}\left(\sum_{j=1}^{K_l} j \mathbf 1(T^{N_l-j}A_l)-
 \sum_{j=n+1}^{n+K_l}(j-n)\mathbf 1(T^{N_l-j}A_l)\right),
 \end{align*} 
 hence 
 \[
  \abs{\widetilde{g_l}-\widetilde{g_l}\circ T^n}\leqslant\frac{N_l^{1/p}}{K_l^{1/2+1/p}}\left(\sum_{j=1}^{K_l} j
  \mathbf 1(T^{N_l-j}A_l)+\sum_{j=n+1}^{n+K_l} (j-n) \mathbf 1(T^{N_l-j-n}A_l)\right),
 \]
 and the following upper bound follows:
 \[
 \mathbb E\abs{\widetilde{g_l}-\widetilde{g_l}\circ T^n}^p\leqslant 2^{p-1}\sum_{j=1}^{K_l}j^pK_l^{-1-p/2}\leqslant 
 2^{p-1}K_l^{p/2}.
 \]
 Treating in a similar manner the function $g_l-\widetilde{g_l}$, we observe that 
 the following inequality holds:
 \begin{equation}\label{moment_cobord}
 \norm{g_l-g_l\circ T^n}_p\leqslant C_p\begin{cases}
                                        nK_l^{-1/2}&\mbox{ if }K_l>n\\
                                        K_l^{1/2}&\mbox{ otherwise,}
                                       \end{cases}  
 \end{equation}
 where $C_p$ depends only on $p$ (neither on $n$, nor on $l$). For a fixed integer $n$, 
 we denote by $i(n)$ the unique integer satisfying the inequalities $K_{i(n)}\leqslant n
 <K_{i(n)+1}$. 
 
 By \eqref{moment_cobord}, we have 
 \begin{align*}
  \norm{g-g\circ T^n}_p&\leqslant \sum_{l=1}^{+\infty}\norm{g_l-g_l\circ T^n}_p\\
  &\leqslant C_p\left(\sum_{l=1}^{i(n)-1}K_l^{1/2}+K_{i(n)}^{1/2}+nK_{i(n)+1}^{-1/2}
  +\sum_{l=i(n)+2}^{+\infty}nK_l^{-1/2}\right)\\
  &\leqslant 3C_p\sqrt n+C_p\sqrt n\sum_{l= i(n)+1}^{+\infty}\frac{K_l}{K_{l+1}^{1/2}}\\
  &\leqslant C_p\left(3+\sum_{l=1}^{+\infty}\frac{K_l}{K_{l+1}^{1/2}}\right)\sqrt n,
 \end{align*}
 where we used \eqref{lac3} in the second inequality and condition \eqref{lac4} ensures 
 finiteness of the right hand side in this inequality.
 
 This concludes the proof of Theorem~\ref{counter_example}.
  \end{proof}
  
 \textbf{Acknowledgements.} \footnote{The final publication is available at Springer via 
\url{http://dx.doi.org/10.1007/s10959-015-0633-9}}
 The author would like to thank the referee for helpful comments which not only 
 improved the presentation of the paper, but also the results of the 
 initial version of Theorem~\ref{theorem_tau_dependent} and 
 Corollary~\ref{corollary_alpha_mixing}.

 The author also thanks Alfredas Ra\v{c}kauskas, Charles Suquet and 
 Dalibor Voln\'y for useful discussions and many valuable remarks and comments.

  \begin{appendices}
   \section{Appendix}
   
   For the reader's convenience, we state deviation inequalities for $\tau$-dependent 
   and $\rho$-mixing sequences.
   
   \begin{nota}
  If $(f\circ T^j)_{j\geqslant 0}$ is a (strictly stationary) 
  sequence of random variables, we define 
  \begin{equation}\label{dfn_s_n}
  s_N^2(f):=\sum_{i=1}^N\sum_{j=1}^N\abs{\operatorname{Cov}(f\circ T^i,
  f\circ T^j)}.
  \end{equation}
  \end{nota}

  \begin{nota}\label{notation_G}
  Let $Y$ be an integrable random variable. We denote by $G_{Y}$
  the generalized inverse of $x\mapsto \int_0^xQ_Y(u)\mathrm du$.
  \end{nota}
   
  The following Fuk-Nagaev inequality was established in Theorem~2 of \cite{MR2105738}. 
   
  \begin{theo}\label{Theorem_tail_inequality_tau_dep}
  Let $(f\circ T^j)_{j\geqslant 0}$ be a strictly stationary sequence of centered and 
  square integrable random variables. Let $R:=((\tau/2)^{-1}\circ G_{f}^{-1}
  )Q_{f}$ and $S=R^{-1}$. For any $\lambda>0$, any integer $N\geqslant 1$ 
  and any $r\geqslant 1$, 
  \begin{equation}\label{tail_inequality_tau_dep}
  \mu\ens{\max_{1\leqslant i\leqslant N}\abs{S_i(f)}\geqslant 
  5\lambda}\leqslant 4\left(1+\frac{\lambda^2}{rs_N^2(f)}\right)^{-r/2}
  +\frac{4N}\lambda\int_0^{S(\lambda/r)}Q_{f}(u)\mathrm du,
  \end{equation}
  and
  \begin{equation}\label{control_variance_tau}
  s_N^2(f)\leqslant 4N\int_0^{\norm{f}_1}(\tau/2)^{-1}(u)
  Q_{f}\circ G_{f}(u)\mathrm du.
  \end{equation}
  \end{theo}

  For $\rho$-mixing sequences, Shao (Theorem~1.2, \cite{MR1334179}) 
  showed the following inequality. 

  \begin{theo}\label{tail_inequality_Shao}
  Let $(f\circ T^j)_{j\geqslant 0}$ be a strictly stationary sequence of centered 
  random variables and $q\geqslant 2$. Then there exists a constant $K$ depending 
  only on $q$ and the sequence 
  $(\rho(n))_{n\geqslant 1}$ such that for each integer $N$ and $x>0$, 
  \begin{multline}\label{Shao}
  \mu\ens{\max_{1\leqslant i\leqslant N}\abs{S_i(f)}\geqslant x}
  \leqslant N\mu\ens{\abs f\geqslant A}+\\
  Kx^{-q}\left(N^{q/2}\exp\left(K\sum_{i=0}^{[\log N]}\rho(2^i)\right)\norm{f}_2^q
  +N\exp\left(K\sum_{i=0}^{[\log N]}\rho^{2/q}(2^i)\right)\norm{f
  \mathbf 1\ens{\abs f\leqslant A}}_q^q\right).
  \end{multline}
  where $A$ satisfies 
  \begin{equation}\label{condition_on_A}
   2N\cdot\mathbb E\left[\abs f \mathbf 1\ens{\abs f\geqslant A}\right]\leqslant x.
  \end{equation}
  \end{theo}

 \end{appendices}
 
\def\polhk\#1{\setbox0=\hbox{\#1}{{\o}oalign{\hidewidth
  \lower1.5ex\hbox{`}\hidewidth\crcr\unhbox0}}}\def\cprime{$'$}
\providecommand{\bysame}{\leavevmode\hbox to3em{\hrulefill}\thinspace}
\providecommand{\MR}{\relax\ifhmode\unskip\space\fi MR }
\providecommand{\MRhref}[2]{%
  \href{http://www.ams.org/mathscinet-getitem?mr=#1}{#2}
}
\providecommand{\href}[2]{#2}

\end{document}